\documentclass[11pt]{amsart}%
\usepackage{amsmath}
\usepackage{amssymb}
\usepackage{amsthm}
\usepackage[mathscr]{eucal}
\usepackage{mathrsfs}
\usepackage{psfrag}
\usepackage{fullpage}
\usepackage{color}
\usepackage{eucal}
\usepackage[dvips]{graphicx}
\usepackage{amsfonts}%
\usepackage{amsaddr}
\providecommand{\U}[1]{\protect\rule{.1in}{.1in}}
\newtheorem{proposition}{Proposition}[section]
\newtheorem{theorem}[proposition]{Theorem}
\newtheorem{corollary}[proposition]{Corollary}
\newtheorem{lemma}[proposition]{Lemma}

\newtheorem{definition}[proposition]{Definition}
\newtheorem{remark}[proposition]{Remark}

\newtheorem{condition}[proposition]{Condition}

\numberwithin{equation}{section}
\numberwithin{proposition}{section}

\newcommand{\eps}{\epsilon}
\newcommand{\R}{\mathbb{R}}

\newcommand{\one}{\mathbf{1}}

\newcommand{ \parbar}[1]{ { \left( #1 \right)} }

\numberwithin{equation}{section}
\numberwithin{proposition}{section}

\begin{document}

\title{Fluctuation analysis and short time asymptotics for multiple scales diffusion processes}

\author{Konstantinos Spiliopoulos}
\address{Department of  Mathematics and Statistics\\
Boston University, Boston, MA, 02215}
\email{kspiliop@math.bu.edu}

\date{\today}

\maketitle

\begin{abstract}
We consider the limiting behavior  of fluctuations of small noise diffusions with multiple scales around their homogenized deterministic limit. We allow full dependence of the coefficients on the slow and fast motion. These processes arise naturally when one is interested in short time asymptotics of multiple scale diffusions. We do not make periodicity assumptions, but we impose conditions on the fast motion to guarantee ergodicity. Depending on the order of interaction between the fast scale and the size of the noise we get different behavior. In certain cases additional drift terms arise in the limiting process, which are explicitly characterized.  These results provide a better approximation to the limiting behavior of such processes when compared to the law of large numbers homogenization limit.
\end{abstract}

\textbf{Keywords}: Fluctuations Analysis; Central Limit Theorem; Multiscale Diffusion Processes.

\textbf{AMS}: 60F05, 60F17, 60G17, 60J60

\section{Introduction}
Consider the $m+(d-m)$ dimensional process $(X^{\epsilon},Y^{\epsilon})=\{(X^{\epsilon}_{s},Y^{\epsilon}_{s}), 0\leq s\leq T\}$ satisfying the system of stochastic differential equations (SDE's)

\begin{eqnarray}
dX^{\epsilon}_{s}&=&\left[  \frac{\epsilon}{\delta}b\left(  X^{\epsilon}_{s}%
,Y^{\epsilon}_{s}\right)+c\left(  X^{\epsilon}_{s}%
,Y^{\epsilon}_{s}\right)\right]   ds+\sqrt{\epsilon}%
\sigma\left(  X^{\epsilon}_{s},Y^{\epsilon}_{s}\right)
dW_{s},\nonumber\\
dY^{\epsilon}_{s}&=&\frac{1}{\delta}\left[  \frac{\epsilon}{\delta}f\left(  X^{\epsilon}_{s}%
,Y^{\epsilon}_{s}\right)  +g\left(  X^{\epsilon}_{s}%
,Y^{\epsilon}_{s}\right)\right] ds+\frac{\sqrt{\epsilon}}{\delta}\left[
\tau_{1}\left(  X^{\epsilon}_{s},Y^{\epsilon}_{s}\right)
dW_{s}+\tau_{2}\left(  X^{\epsilon}_{s},Y^{\epsilon}_{s}\right)dB_{s}\right], \label{Eq:Main}\\
X^{\epsilon}_{0}&=&x_{0},\hspace{0.2cm}Y^{\epsilon}_{0}=y_{0}\nonumber
\end{eqnarray}
where $\delta=\delta(\epsilon)\downarrow0$ as $\epsilon\downarrow0$ and $(W_{s}, B_{s})$ is a $2\kappa-$dimensional standard Wiener process. Assumptions on the coefficients $b(x,y),c(x,y),\sigma(x,y), f(x,y), g(x,y), \tau_{1}(x,y)$ and $\tau_{2}(x,y)$ are given in Condition \ref{A:Assumption1}. The purpose of this paper is obtain the limiting behavior of the fluctuations process
\begin{equation}
\eta^{\epsilon}_{t}=\frac{X^{\epsilon}_{t}-\bar{X}_{t}}{\beta^{\epsilon}},\quad \textrm{as }\epsilon\downarrow 0,
\end{equation}
where $\beta^{\epsilon}$ is the appropriate normalization constant and $\bar{X}$ is the homogenization limit of $X^{\epsilon}$ as $\epsilon,\delta\downarrow 0$. We are interested in the limiting behavior of $\{\eta^{\epsilon}_{\cdot}, \epsilon>0\}$ in the following two cases
\begin{equation}
\lim_{\epsilon\downarrow0}\frac{\epsilon}{\delta}=%
\begin{cases}
\infty & \text{Regime 1,}\\
\gamma\in(0,\infty) & \text{Regime 2.}\\
\end{cases}
\label{Def:ThreePossibleRegimes}%
\end{equation}
Depending on the regime of interaction both the law of large numbers limit $\bar{X}$ and the limit of the correction process $\left\{\eta^\epsilon_{t}, \epsilon>0, t\in[0,T]\right\}$ are different. It is important to note that we do not make  compactness assumptions for the fast motion such as periodicity.

The novelty of this work lies on the consideration of systems of slow and fast motion with coefficients fully dependent on the slow and fast motion in the whole space. The lack of compactness makes the analysis more complicated compared to the periodic case. At this point we make use of the recent results in \cite{PardouxVeretennikov1, PardouxVeretennikov2} that allow to pose and study Poisson equations on the whole space. Moreover, it is interesting to note that the interaction of $\epsilon$ with $\delta$ has several consequences for the limiting behavior. Not only, is the limit different for each regime of interaction, but additional terms may appear in the limiting equation that are not present in the absence of multiscale features.

Models like (\ref{Eq:Main}) can be thought as perturbations of an underlying deterministic dynamical systems, $\dot{X}=\bar{\lambda}(X)$, by small noise and multiple scales. For example, if $\bar{\lambda}(x)$ is defined as the integral of a given function with respect to a measure $\mu$, then this dynamical system can be thought of as a small noise perturbation of a system of slow and fast motion, where the integrating measure $\mu$ is the invariant measure of the fast motion.
Such models also arise when one deals with mulitple scale systems but the interest is in small time asymptotics. For example, consider a classical system of stochastic differential equations with slow and fast components

\begin{eqnarray}
dX_{s}&=&c\left(  X_{s}%
,Y_{s}\right)   ds+%
\sigma\left(  X_{s},Y_{s}\right)
dW^{(1)}_{s},\nonumber\\
dY_{s}&=&  \frac{1}{\delta^{2}}f\left(  X_{s}%
,Y_{s}\right)  ds+\frac{1}{\delta}\tau\left(  X_{s},Y_{s}\right)
dW^{(2)}_{s}, \nonumber
\end{eqnarray}
where  $W^{(1)}$ and $W^{(2)}$ are correlated Wiener processes. If one is interested in short time asymptotics, it is convenient to rescale time $s\mapsto \epsilon s$, and then the process $(X^{\epsilon}_{s},Y^{\epsilon}_{s})=(X_{\epsilon s},Y_{\epsilon s})$ satisfies (\ref{Eq:Main}) with $b(x,y)=g(x,y)=0$ and $c(x,y)$ replaced by $\epsilon c(x,y)$ .  A related example  with connections to large deviations theory is presented in Section \ref{S:Examples}.

Of course, the history of similar limiting theorems for stochastic dynamical systems is long. Limiting theorems, such as law of large numbers, central limit theorems and large deviations for $X^{\epsilon}$ when $b=0$ and the coefficients $c$ and $\sigma$ are independent of the $y$ variable are available, see for example \cite{Freidlin1978, FWBook}. Cases with averaging effects in periodic or stationary random environments have also been studied for special cases of the system (\ref{Eq:Main}), see \cite{BaierFreidlin,DupuisSpiliopoulos,Freidlin1978, FS, FWBook,Guillin,KlebanerLipster,LiptserPaper,Spiliopoulos2012}. Some law of large numbers and large deviations results in the whole space are available in \cite{PardouxVeretennikov1, PardouxVeretennikov2,Veretennikov,VeretennikovSPA2000}. To the best knowledge of the author, the existing literature does not address the fluctuations analysis done in this paper.

The rest of the paper is organized as follows. In Section \ref{S:Notation} we introduce notation, assumptions and summarize preliminary results. In particular, we present the law of large numbers result, namely the limit of the slow component $X^{\epsilon}$ as $\epsilon\downarrow 0$. We also recall, regularity results on Poisson equations on the whole space \cite{PardouxVeretennikov1,PardouxVeretennikov2} that will be used throughout the paper. These results are necessary in order to study the behavior of correctors in the absence of the periodicity assumption. In Section \ref{S:MainTheorem}, we present our main result. The proof for Regime $2$ is given in Section \ref{S:ProofRegime1}, whereas the proof for Regime $1$ is given in Section \ref{S:ProofRegime2}. The order of consideration of the two regimes is reversed in order to be consistent with the existing large deviations literature \cite{DupuisSpiliopoulos,Spiliopoulos2012} and because Regime $2$ is simpler to analyze than Regime $1$. An example to illustrate our results is presented in Section \ref{S:Examples}. In Section \ref{S:Examples} we also connect the validity of our central limit theorem to the second derivative of the related large deviations action functional (obtained in \cite{Spiliopoulos2012}) in a simple case.

\section{Notation, assumptions and preliminary results}\label{S:Notation}

In this section we present preliminary results that will be used throughout the paper. However, first we need to establish notation and pose the assumptions on the coefficients.

For notational convenience we denote by $\mathcal{Y}=\mathbb{R}^{d-m}$ the state space of the fast motion.  The functions $b , c,f,g,\sigma,\tau_{1}$ and $\tau_{2}$ satisfy the following conditions:
\begin{condition}
\label{A:Assumption1}
\begin{enumerate}
\item The diffusion matrix $\tau_{1}\tau_{1}^{T}+\tau_{2}\tau_{2}^{T}$ is uniformly nondegenerate.
\item Let $h$ be any of the functions $b$ or $c$. We assume that $h(\cdot,y)\in C^{2}(\mathbb{R}^{m})$ for all $y\in\mathcal{Y}$, $\frac{\partial^{2}h}{\partial y^{2}}\in C\left(\mathbb{R}^{m},\mathcal{Y}\right)$, $h(x,\cdot)\in C^{\alpha}\left(\mathcal{Y}\right)$ uniformly in $x\in\mathbb{R}^{m}$ for some $\alpha\in(0,1)$ and that there exist $K$ and $q$ such that
    \[
    \sum_{i=0}^{2}\left|\frac{\partial^{i} h}{\partial x^{i}}(x,y)\right|\leq K\left(1+|y|^{q}\right).
    \]
\item For every $N>0$ there exists a constant $C(N)$ such that for all $x_{1},x_{2}\in\mathbb{R}^{m}$ and $|y|\leq N$, the diffusion matrix $\sigma$ satisfies
\[
\left|\sigma(x_{1},y)-\sigma(x_{2},y)\right|\leq C(N)|x_{1}-x_{2}|.
\]
Moreover, there exists $K>0$ and  $q>0$ such that
\[
|\sigma(x,y)|\leq K (1+|x|^{1/2})(1+|y|^{q}).
\]
\item The functions $f(x,y), \tau_{1}(x,y)$ and $\tau_{2}(x,y)$ are  $C^{2,2+\alpha}_{b}(\mathbb{R}^{m}\times\mathcal{Y})$ with $\alpha\in(0,1)$. Namely, they have two bounded derivatives in $x$ and $y$, with all partial derivatives being H\"{o}lder continuous, with exponent $\alpha$, with respect to $y$, uniformly in $x$.
\item In the case of Regime $1$, function $g$ is assumed to have the smoothness and growth conditions of $b$ and $c$. In the case of Regime $2$, function $g$ is assumed to have the smoothness and growth conditions of $f,\tau_{1}$ and $\tau_{2}$.
\end{enumerate}
\end{condition}

\begin{definition}
\label{Def:ThreePossibleOperators} For
$(x,y)\in\mathbb{R}^{m}\times\mathcal{Y}$
and for Regime $i=1,2$ defined in (\ref{Def:ThreePossibleRegimes})
we define the operators $\mathcal{L}_{x}^{i}$ with domain of definition $\mathcal{D}(\mathcal{L}_{z,x}^{i})=\mathcal{C}%
^{2}(\mathcal{Y})$ as follows
\begin{align}
\mathcal{L}_{x}^{1}  &  =f(x,\cdot)D_{y}+\frac{1}{2}\textrm{tr}\left[\left(\tau_{1}\tau_{1}^{T}+\tau_{2}\tau_{2}^{T}\right)(x,\cdot)D^{2}_{y}\right] \nonumber\\
\mathcal{L}_{x}^{2}  &  =\left[  \gamma
f(x,\cdot)+g(x,\cdot)\right]
D_{y} +\gamma\frac{1}{2}\textrm{tr}\left[\left(\tau_{1}\tau_{1}^{T}+\tau_{2}\tau_{2}^{T}\right)(x,\cdot)D^{2}_{y}\right]\nonumber
\end{align}
\end{definition}

In order to guarantee existence of a unique invariant measures associated to the operators $\mathcal{L}_{x}^{i},i=1,2$ just defined, we need to impose, apart from the non-degeneracy condition on the diffusion coefficient, the following:

\begin{condition}\label{A:LyapunovCondition}
We assume that
\begin{enumerate}
\item{Regime $1$: $\lim_{|y|\rightarrow\infty}\sup_{x\in\mathbb{R}^{m}}f(x,y)\cdot y=-\infty$.}
\item{Regime $2$: $\lim_{|y|\rightarrow\infty}\sup_{x\in\mathbb{R}^{m}}(\gamma f(x,y)+g(x,y))\cdot y=-\infty$.}
    \end{enumerate}
\end{condition}

To this end, let us denote by $\mu_{i}(dy|x)$ the unique invariant measures corresponding to the operators $\mathcal{L}_{x}^{i}$. For Regime $1$, we additionally assume

\begin{condition}
\label{A:Assumption2} Under Regime 1, we assume the centering condition for the drift term $b$:
\[
\int_{\mathcal{Y}}b(x,y)\mu_{1}(dy|x)=0.
\]
\end{condition}

Next we recall some regularity results from \cite{PardouxVeretennikov1,PardouxVeretennikov2} (Lemma 4 in \cite{PardouxVeretennikov1} and Theorem 3 in \cite{PardouxVeretennikov2}) for Poisson equations on the whole space, appropriately phrased to cover our case of interest.
\begin{theorem} \label{T:RegularityPoisson}
Let Conditions \ref{A:Assumption1} and \ref{A:LyapunovCondition} be satisfied. Assume that $G(x,y)\in C^{2,\alpha}\left(\mathbb{R}^{m}, \mathcal{Y}\right)$,
\[
\int_{\mathcal{Y}}G(x,y)\mu_{i}(dy|x)=0.
\]
and that for some positive constants $K$ and $q$,
\[
    \sum_{i=0}^{2}\left|\frac{\partial^{i} G}{\partial x^{i}}(x,y)\right|\leq K\left(1+|y|^{q}\right)
\]
Then, the solution to the Poisson equation
\begin{eqnarray}
& &\mathcal{L}_{x}^{i}u(x,y)=-G(x,y),\quad\int_{\mathcal{Y}}%
G(x,y)\mu_{i}(dy|x)=0 \label{Eq:CellProblem}
\end{eqnarray}
satisfies $u(\cdot,y)\in C^{2}$ for every $y\in\mathcal{Y}$, $\partial_{x}^{2}u\in C\left(\mathbb{R}^{m}\times\mathcal{Y}\right)$ and there exist positive constants $K'$ and $q'$ such that
\[
    \sum_{i=0}^{2}\left|\frac{\partial^{i} u}{\partial x^{i}}(x,y)\right|+\left|\frac{\partial^{2} u}{\partial x\partial y}(x,y)\right|\leq K'\left(1+|y|^{q'}\right)
\]
\end{theorem}

\begin{remark}\label{R:LyapunovCondition}
It seems plausible that Condition \ref{A:LyapunovCondition} can be weekend and replaced by less strong assumptions that still guarantee existence of an invariant measure. As an example, assume for every $x\in\mathbb{R}^{m}$
\begin{enumerate}
\item{Regime $1$: $\limsup_{|y|\rightarrow\infty}\left[f(x,y)\cdot y+ \left[\tau_{1}(x,y)\tau_{1}^{T}(x,y)+\tau_{2}(x,y)\tau_{2}^{T}(x,y)\right]\right]<0$.}
\item{Regime $2$: $\limsup_{|y|\rightarrow\infty}\left[(\gamma f(x,y)+g(x,y))\cdot y+\gamma \left[\tau_{1}(x,y)\tau_{1}^{T}(x,y)+\tau_{2}(x,y)\tau_{2}^{T}(x,y)\right]\right]<0$.}
    \end{enumerate}
The results that we use from \cite{PardouxVeretennikov1,PardouxVeretennikov2} hold under the assumed there Condition \ref{A:LyapunovCondition}. However, an examination of the proofs of the quoted results from those papers, shows that weaker condition, as the aforementioned one, can be used. Moreover, we note here that under such conditions, the standard Lyapunov type condition for existence of an invariant measure of \cite{Hasminskii} is satisfied (see Example 3.9 of \cite{Hasminskii}).
\end{remark}

The solution to the Poisson equation has the representation

\begin{equation}
u(x,y)=\int^{\infty}_{0}E_{x,y} G\left(x,Y^{i,x}_{t}\right)dt
\end{equation}
where $Y^{i,x}_{t}$ is the Markov process with infinitesimal generator $\mathcal{L}_{x}^{i}$.

Letting for each $l\in\{1,\ldots,m\}$, $G=b_{\ell}$, we then denote by
$\chi=(\chi_{1},\ldots,\chi_{m})$ the solution to (\ref{Eq:CellProblem}). This is the solution to the so-called cell problem in periodic homogenization, e.g., \cite{BLP}.

It will become useful to define functions $\lambda_{i}(x,y)$ and $\bar{\lambda}_{i}(x)$, $i=1,2$, as follows:

\begin{definition}
\label{Def:ThreePossibleFunctions} For $(x,y)\in\mathbb{R}^{m}\times\mathcal{Y}$
and for Regime $i=1,2$
defined in (\ref{Def:ThreePossibleRegimes}) we define the functions $\lambda_{i}(x,y):\mathbb{R}%
^{m}\times\mathcal{Y}\rightarrow\mathbb{R}^{m}$ by
\begin{align}
\lambda_{1}(x,y)  &  = c(x,y)+\frac{\partial\chi}{\partial y}(x,y) g(x,y) \nonumber\\
\lambda_{2}(x,y)  &  =\gamma b(x,y)+c(x,y)\nonumber
\end{align}
where $\chi=(\chi_{1},\ldots,\chi_{m})$ is defined by (\ref{Eq:CellProblem}) with $G=b_{\ell}$. Set \begin{equation*}
 \bar{\lambda}_{i}(x)=\int_{\mathcal{Y}}\lambda_{i}(x,y)\mu_{i}(dy|x).
\end{equation*}
\end{definition}

Due to Condition \ref{A:Assumption1}, Proposition 1 and Theorem 3 of \cite{PardouxVeretennikov2}, we get that $\lambda_{i}(x,y), i=1,2$ are once continuously differentiable with respect to the x-variable. Moreover, by Condition \ref{A:Assumption1} and Theorem 2 of \cite{PardouxVeretennikov2}, we also have that the invariant measures $\mu_{i}(dy|x)$ are once continuously differentiable with respect to $x$. Thus, we infer that $\bar{\lambda}_{i}\in C^{1}(\mathbb{R}^{m})$. For $x\in\mathbb{R}^{m}$, let $\bar{X}^{i}_{s}$ be the solution to the ordinary differential equation
\begin{equation}
\bar{X}^{i}_{t}=x+\int_{0}^{t}  \bar{\lambda}_{i}(\bar{X}^{i}_{s})ds.\label{Eq:LimitingODE}
\end{equation}

We may write $\bar{X}^{i}_{t}(x)$ if we want to emphasize the dependence on the initial point. Based on the results in \cite{Spiliopoulos2012}, we obtain the following theorem, which essentially is the law of large numbers for (\ref{Eq:Main}). The proof follows as in \cite{Spiliopoulos2012}, so we only include a short note.

\begin{theorem}
\label{T:LLN} Consider any $x_{0}\in\mathbb{R}^{m}$ and any $T>0$. Assume Conditions \ref{A:Assumption1} and \ref{A:LyapunovCondition}. In addition, in Regime 2 assume
Condition \ref{A:Assumption2}. Then, we have that for all  $\eta>0$ and $i=1,2$
\begin{equation}
\lim_{\epsilon \to 0}\mathbb{P}\left\{ \sup_{0\leq t\leq T}\left|X^{\epsilon}_{t}-\bar{X}^{i}_{t} (x_0) \right|>\eta\right\}=0
, \quad T>0.\label{Eq:LLN}
\end{equation}
\end{theorem}
\begin{proof}[Sketch of the proof]
Under our assumptions, Theorem 3.2 in \cite{Spiliopoulos2012} guarantees weak convergence of $X^{\epsilon}_{\cdot}$ to $\bar{X}^{i}_{\cdot}$ in $\mathcal{C}([0,T];\mathbb{R}^{m})$ for any $T>0$. Since here, the limiting process $\bar{X}^{i}_{\cdot}$ is deterministic, we obtain the convergence in probability claim of the theorem. Also, due to our assumptions, the limiting ODE's in (\ref{Eq:LimitingODE}) are well defined and have a unique solution in their corresponding regime.
\end{proof}

\section{Main theorem}\label{S:MainTheorem}

In this section we describe our main results. Proofs are in the subsequent sections. A term that will appear frequently in the analysis is
\begin{equation}\label{Eq:CorrectionTerm}
\Gamma^{\epsilon,\delta}_{t}=\int_{0}^{t}\left(\lambda_{i}\left(X^{\epsilon}_{s},Y^{\epsilon}_{s}\right)-\bar{\lambda}_{i}(X^{\epsilon}_{s})\right)ds
\end{equation}
We investigate its dependence on $\epsilon$ and $\delta$ by considering the auxiliary Poisson equation
\begin{eqnarray}
& & \mathcal{L}_{x}^{i}\Phi_{i}(x,y)=-\left(\lambda_{i}  \left( x,y\right) - \bar{\lambda}_{i}(x)\right),\quad\int_{\mathcal{Y}}%
\Phi_{i}(x,y)\mu_{i}(dy|x)=0,\hspace{0.1cm} \label{Eq:CellProblemCLT}\\
& & \Phi_{i} \textrm{ grows at most polynomially in }y\textrm{ as } |y|\rightarrow\infty\nonumber
\end{eqnarray}
for $i=1,2$. By construction, the right hand side of the PDE averages to zero. Therefore, Theorem \ref{T:RegularityPoisson} implies that the function $\Phi_{i}(x,y)$ is uniquely defined and has the smoothness properties of the solution $u$ to (\ref{Eq:CellProblem}), if the right hand side has the appropriate smoothness assumptions. It turns out that Condition \ref{A:Assumption1} guarantees that this is the case. More details will be discussed in the corresponding proofs.

For notational convenience, we shall denote by
\[
\bar f_{i}(x) = \int_{ \mathcal{Y} } f (x,y) \mu^i (dy|x).
\]
the average of a function $f : \R^m \times \mathcal{Y} \to \R^m$ with respect to $\mu^i$.

For Regimes $i=1,2$ we define $J_{i}$ and $q_{i}$ as follows:

\begin{align}
J_{1} (x,y)&=\left[\frac{\partial \Phi_{1}}{\partial y} g\right](x,y), \\
q_{1}(x,y)&= \left[\left(\sigma+\frac{\partial\chi}{\partial y}\tau_{1}\right)\left(\sigma+\frac{\partial\chi}{\partial y}\tau_{1}\right)^{T}+\left(\frac{\partial\chi}{\partial y}\tau_{2}\right)\left(\frac{\partial\chi}{\partial y}\tau_{2}\right)^{T}\right](x,y), \nonumber
\end{align}
and
\begin{align}
J_{2}( x,y)&=\left[b-\frac{1}{\gamma}\left(\lambda_{2}-\bar{\lambda}_{2}+\frac{\partial \Phi_{2}}{\partial y} g\right)\right](x,y), \label{Eq:Jfunction}\\
q_{2}(x,y)&= \left[\left(\sigma+\frac{\partial\Phi_{2}}{\partial y}\tau_{1}\right)\left(\sigma+\frac{\partial\Phi_{2}}{\partial y}\tau_{1}\right)^{T}+\left(\frac{\partial\Phi_{2}}{\partial y}\tau_{2}\right)\left(\frac{\partial\Phi_{2}}{\partial y}\tau_{2}\right)^{T}\right](x,y). \nonumber
\end{align}

With these definitions in hand, we are ready to state our results.

\begin{theorem}\label{T:CLT2}
Let $T>0$. Consider the solution to the equation (\ref{Eq:Main}). Assume Regime $i=1,2$ and let Conditions \ref{A:Assumption1},\ref{A:LyapunovCondition} and \ref{A:Assumption2} holding. Set $\theta_{1}^{\epsilon}=\frac{\delta}{\epsilon}$ and $\theta_{2}^{\epsilon}=\frac{\epsilon}{\delta}-\gamma$. Let $\ell_{i}\in [0,\infty]$ with $i=1,2$ be given by
 $$ \ell_{i} = \lim_{ \eps \to 0 } \frac{\sqrt{\epsilon}}{\theta^{\epsilon}_{i}}$$
and
\begin{equation*}
\beta^{\epsilon}_{i} =
\begin{cases}
\theta_{i}^{\epsilon} & ,\ell_{i}=0,\\
\sqrt{\epsilon} & , \ell_{i} \in (0,\infty]
\end{cases}
\end{equation*}
The process
$$\eta^{\epsilon}_{t}=\frac{X^{\epsilon}_{t}-\bar{X}^{i}_{t}}{\beta^{\epsilon}_{i}}$$ converges weakly in the space of continuous functions in $\mathcal{C}\left([0,T];\mathbb{R}^{m}\right)$
to the solution of the Ornstein-Uhlenbeck type process
\begin{eqnarray}
 d\eta_{t}&=&D\bar{\lambda}_{i}(\bar{X}^{i}_{t}(x_{0}))\eta_{t}dt+\left[ \ell^{-1}_{i} \one (\ell_{i} \in(0,\infty] ) + \one (\ell_{i} =0 ) \right ]\bar{J}_{i}(\bar{X}^{i}_{t}(x_{0}))dt+\nonumber\\
  & &\quad +\one \left(\ell_{i}\neq 0 \right)\bar{q}^{1/2}_{i}(\bar{X}^{i}_{t}(x_{0}))d\tilde{W}_{t}\nonumber\\
\eta_{0}&=&0.\label{Eq:LimitingProcess}
\end{eqnarray}
where $\tilde{W}$ is an $k-$dimensional standard Wiener process.
\end{theorem}

The following remark is of interest.

\begin{remark}
Note that if $\ell_{i}\in[0,\infty)$, then the limiting SDE (\ref{Eq:LimitingProcess}) has the additional drift term $\bar{J}_{i}(\bar{X}^{i}_{t}(x_{0}))$, which vanishes from (\ref{Eq:LimitingProcess}) only in the case $\ell_{i}=\infty$.  It is easy to see that $\ell_{1}=\infty$ if $\delta=o(\epsilon^{3/2})$ (Regime 1) and in the case of Regime 2, if $\delta=\frac{1}{\gamma}\epsilon$, then $\ell_{2}=\infty$.
\end{remark}

Notice now that it is not difficult to solve  the SDE (\ref{Eq:LimitingProcess}) explicitly. In particular, letting for $x \in \R^m$, $\Psi^i_{x}$  be the linearization of $\bar X^{i}$ along the orbit of $x$:
\begin{equation}
 \frac{d}{dt}\Psi^i_{x}(t)=D\bar{\lambda}^{i}( \bar X^i_t )\Psi^i_{x}(t), \text{  } \Psi^{i}_{x}(0)=x
\end{equation}
where $D\bar{\lambda}^{i}$ is the Jacobian matrix of $\bar{\lambda}^{i}$ and the defining

\begin{align}
\Theta^{i}_{x_0} (t) &= \Psi_{x_{0}}^i (t) \int_{0}^{t} \left[\Psi_{x_{0}}^i (s)\right]^{-1} \bar{q}^{1/2}_i( \bar X_s^{i} )  d\tilde{W}_{s} , \quad t \geq 0, \label{eqn: thetadef_0}
\end{align}
and
\begin{equation}
H^{i}_{x_0} (t)= \Psi_{x_{0}}^i (t)
\int_0^t \left[\Psi_{x_{0}}^i (s)\right]^{-1} \bar J_{i} \left(  \bar X_{s}^{i} \right)ds
\end{equation}
we obtain by Duhamel's principle that
\begin{align}
\eta_{t}^i (\ell_{i}) =  \Theta^i_{x_0} (t) \one \left(\ell_{i} \ne 0 \right)+  H^i_{x_0} (t)   \left[ \ell^{-1}_{i} \one (\ell_{i} \in(0,\infty] ) + \one (\ell_{i} =0 ) \right ].   \label{eqn: Eta23}
\end{align}

In Sections \ref{S:ProofRegime1} and \ref{S:ProofRegime2} we prove Theorem \ref{T:CLT2}.

\section{Proof of Theorem \ref{T:CLT2} for Regime 2.}\label{S:ProofRegime1}
In this section we present the proof of Theorem \ref{T:CLT2} in the case of Regime 2, i.e., when $\frac{\epsilon}{\delta}\rightarrow \gamma\in(0,\infty)$. For notational convenience we omit emphasizing the dependence of the involved functions on Regime $2$, i.e., we do not write the subscript $2$. Namely, we shall write $\Phi$, instead of $\Phi_{2}$, for the solution of the Poisson equation (\ref{Eq:CellProblemCLT}) and similarly for the functions $J,q,\lambda$, the operator $\mathcal{L}_{x}$ and the measure $\mu$.

Next, we write the equation that $\eta^{\epsilon}=\left(X^{\epsilon}-\bar{X}\right)/\beta^{\epsilon}$ satisfies in a convenient way. The first step is a representation formula for (\ref{Eq:CorrectionTerm}). We have the following lemma.
\begin{lemma}\label{L:Representationfor_I}
Assume Conditions \ref{A:Assumption1} and \ref{A:LyapunovCondition}. The following hold
\begin{enumerate}
\item{The solution to the Poisson equation (\ref{Eq:CellProblemCLT}) satisfies the conclusions of Theorem \ref{T:RegularityPoisson}.}
\item{For every $\epsilon,\delta>0$ we have the representation
\begin{align*}
\Gamma^{\epsilon,\delta}_{t}&=\int_{0}^{t}\left[\lambda  \left( X_{s}^{\epsilon},Y_{s}^{\epsilon}\right) - \bar{\lambda}(X_{s}^{\epsilon})\right]  ds\nonumber\\
&=
\left(\frac{\epsilon}{\delta}-\gamma\right)\int_{0}^{t}(J-b)\left( X_{s}^{\epsilon},Y_{s}^{\epsilon}\right)  ds-\delta\left(\Phi \parbar{ X^\eps _t, Y_{s}^{\epsilon}  }-\Phi \parbar{ X^{\epsilon}_0, Y^{\epsilon}_{0}  }\right)+\int_{0}^{t}\mathcal{R}^{\epsilon}\left( X_{s}^{\epsilon},Y_{s}^{\epsilon}\right)ds\nonumber\\
&\hspace{0.3cm}
+\sqrt{\epsilon}\int_{0}^{t} \left(\delta\frac{\partial \Phi}{\partial x}\sigma+\frac{\partial \Phi}{\partial y}\tau_{1}\right)\parbar{ X^\eps _s, Y^\eps _s }dW_s+\sqrt{\epsilon}\int_{0}^{t} \frac{\partial \Phi}{\partial y}\tau_{2}\parbar{ X^\eps _s, Y^\eps _s  }dB_s,\nonumber
\end{align*}
where
\[
\mathcal{R}^{\epsilon}(x,y)=\left[\epsilon \frac{\partial \Phi}{\partial x}b+\delta \frac{\partial \Phi}{\partial x}c+\frac{\epsilon\delta}{2}\textrm{tr}\left(\frac{\partial^{2} \Phi}{\partial x^{2}}\sigma\sigma^{T}\right)
+\epsilon\textrm{tr}\left(\frac{\partial \Phi}{\partial x\partial y}\sigma\tau_{1}^{T}\right)\right](x,y).
\]}
\end{enumerate}
\end{lemma}

\begin{proof}
Part (i). We need to verify that the right hand side of (\ref{Eq:CellProblemCLT}), i.e., $G(x,y)=\lambda(x,y)-\bar{\lambda}(x)$ satisfies the assumptions of Theorem \ref{T:RegularityPoisson}. Keeping in mind that in Regime 2, we have $\lambda(x,y)=\gamma b(x,y)+c(x,y)$, the smoothness and growth conditions are satisfied for $\lambda(x,y)$ due to Condition \ref{A:Assumption1}. For $\bar{\lambda}(x)=\int_{\mathcal{Y}}\lambda(x,y)\mu(dy|x)$ the same is true if the invariant measure $\mu(dy|x)$ is appropriately smooth. By part (iv) of Condition \ref{A:Assumption1}, this follows by Theorem 1 of \cite{PardouxVeretennikov2}.

Part (ii). By part (i) we can apply the  It\^{o}-Krylov formula  to $\Phi(x,y)=\left(\Phi_{1}(x,y),\cdots,\Phi_{m}(x,y)\right)$ with $(x,y)=\left(X^{\epsilon}_{t},Y^{\epsilon}_{t}\right)$. We obtain
\begin{align*}
 \delta\Phi \parbar{ X^\eps _t, Y^\eps _t } &= \delta\Phi \parbar{ X^\eps _0, Y^\eps _0 }+\int_{0}^{t}\mathcal{R}^{\epsilon}\left( X_{s}^{\epsilon},Y_{s}^{\epsilon}\right)ds+\int_{0}^{t}\mathcal{L}_{X^{\epsilon}_{s}}\Phi\left( X_{s}^{\epsilon},Y_{s}^{\epsilon}\right)ds\nonumber\\
&\hspace{0.3cm}+\int_{0}^{t}\left(\frac{\epsilon}{\delta}-\gamma\right)\left(\frac{\partial \Phi}{\partial y}f
+\frac{1}{2}\textrm{tr}\left[\frac{\partial^{2} \Phi}{\partial y^{2}}\left(\tau_{1}\tau_{1}^{T}+\tau_{2}\tau_{2}^{T}\right)\right]\right)\left( X_{s}^{\epsilon},Y_{s}^{\epsilon}\right)ds\nonumber\\
&\hspace{0.3cm}
+\sqrt{\epsilon}\int_{0}^{t} \left(\delta\frac{\partial \Phi}{\partial x}\sigma+\frac{\partial \Phi}{\partial y}\tau_{1}\right)\parbar{ X^\eps _s, Y^\eps _s }dW_s+\sqrt{\epsilon}\int_{0}^{t} \frac{\partial \Phi}{\partial y}\tau_{2}\parbar{ X^\eps _s, Y^\eps _s  }dB_s
\end{align*}
Then, taking into account that $\Phi$ satisfies the PDE (\ref{Eq:CellProblemCLT}) and that by the definition of $J(x,y)$ by (\ref{Eq:Jfunction})
\begin{eqnarray}
\left(\frac{\partial \Phi}{\partial y}f+\frac{1}{2}\textrm{tr}\left[\frac{\partial^{2} \Phi}{\partial y^{2}}\left(\tau_{1}\tau_{1}^{T}+\tau_{2}\tau_{2}^{T}\right)\right]\right)(x,y)&=&\frac{1}{\gamma}\left(\mathcal{L}_{x}\Phi-\frac{\partial \Phi}{\partial y}g\right)(x,y)\nonumber\\
&=&-\frac{1}{\gamma}\left(\lambda(x,y)-\bar{\lambda}(x)+\frac{\partial \Phi}{\partial y}g(x,y)\right)\nonumber\\
&=&J(x,y)-b(x,y)\nonumber
\end{eqnarray}
we get the claim of the lemma.
\end{proof}

Let us then proceed by rewriting the expression for $\Delta^{\epsilon}_{t}=X^{\epsilon}_{t}-\bar{X}_{t}$. Clearly we have that
\begin{align} \notag
\Delta^{\epsilon}_{t}&=\int_{0}^{t}\left[  \frac{\epsilon}{\delta}b\left(  X_{s}^{\epsilon}
, Y_{s}^{\epsilon}\right)  +c\left(  X_{s}^{\epsilon}%
, Y_{s}^{\epsilon}\right)  -\bar{\lambda}( \bar X _ s )\right]  ds  +\sqrt{\epsilon}\int_{0}^{t}\sigma\left(  Y_{t}^{\epsilon}, Y_{s}^{\epsilon}\right)dW_{s}.\label{eqn: DeltaEqn}
\end{align}

Smoothness of $\bar{\lambda}$ implies via Taylor's theorem that
\[
\bar{\lambda}(x_1)=  \bar{\lambda}(x_{2}) +D_x \bar{\lambda}(x_{2}) (x_{1}-x_{2}) + \Lambda[\bar{\lambda}](x_{1},x_{2}), \quad x_{1},x_{2} \in \R^m,
\]
for some function $\Lambda[\bar{\lambda}]$ such that $|x_1 - x_2|^{-2}\Lambda[\bar{\lambda}] (x_1,x_2)$ is locally bounded. Therefore, we obtain

\begin{align*}
\Delta^\epsilon_{t} =& \int_{0}^{t} D_x \bar{\lambda} ( \bar X_s) \Delta^\eps_{s} ds  + \left(\frac{\epsilon}{\delta}-\gamma\right) \int_{0}^{t}b\left(  X_{s}^{\epsilon},Y_{s}^{\epsilon}\right)ds
+ \int_{0}^{t}\left[\lambda  \left(  X_{s}^{\epsilon},Y_{s}^{\epsilon}\right) - \bar{\lambda}(X_{s}^{\epsilon})\right]  ds \\
&\quad + \int_0^t \Lambda[\bar{\lambda}]\parbar{ \bar X_s,X^\eps_s}ds  + \sqrt{\epsilon}\int_{0}^{t}\sigma  \left(X_{s}^{\epsilon},Y_{s}^{\epsilon}\right) dW_{s}
\end{align*}

Hence, by Lemma \ref{L:Representationfor_I} we get that $\eta^{\epsilon}_{t}=\Delta^\epsilon_{t}/\beta^{\epsilon}$ satisfies
\begin{align}
\eta^\epsilon_{t} &=  \int_{0}^{t} D_x \bar{\lambda} ( \bar X_s) \eta^\eps_{s} ds+
\frac{\left(\frac{\epsilon}{\delta}-\gamma\right)}{\beta^{\epsilon}}\int_{0}^{t}J  \left( X_{s}^{\epsilon},Y_{s}^{\epsilon}\right)  ds\label{Eq:DeltaExpression}\\
&\hspace{0.3cm}-\left(\delta/\beta^{\epsilon}\right)\left(\Phi \parbar{ X^\eps _t, Y^\eps _t   }-\Phi \parbar{ X^\eps _0, Y^\eps _0 }\right)+\nonumber\int_{0}^{t}\frac{1}{\beta^{\epsilon}}\mathcal{R}^{\epsilon}\parbar{ X^\eps _s, Y^\eps _s }ds+ \int_0^t \frac{1}{\beta^{\epsilon}}\Lambda[\bar{\lambda}]\parbar{ \bar X_s,X^\eps_s}ds\\
&\hspace{0.3cm}+\frac{\sqrt{\epsilon}}{\beta^{\epsilon}}\int_{0}^{t} \left(\delta\frac{\partial \Phi}{\partial x}\sigma+\sigma+\frac{\partial \Phi}{\partial y}\tau_{1}\right)\parbar{ X^\eps _s, Y^\eps _s }dW_s+\frac{\sqrt{\epsilon}}{\beta^{\epsilon}}\int_{0}^{t} \frac{\partial \Phi}{\partial y}\tau_{2}\parbar{ X^\eps _s, Y^\eps _s  }dB_s\nonumber
  \end{align}

For the sake of presentation, we split the rest of the proof of the theorem in two subsections. In Subsection \ref{SS:TightnessRegime2}, we prove that the family $\left\{\left(X^{\epsilon}_{\cdot},\eta^{\epsilon}_{\cdot}\right),\epsilon>0\right\}$ is relatively compact in $\mathcal{C}\left([0,T];\mathbb{R}^{m}\right)$. Then, in Subsection \ref{SS:Reg2IdentificationLimit} we identify the limit via martingale arguments. Together with uniqueness of the solution to the limiting equation, Theorem \ref{T:CLT2} follows.

\subsection{Tightness}\label{SS:TightnessRegime2}

We prove tightness of the family $\left\{\eta^{\epsilon}_{\cdot},\epsilon>0\right\}$ making use of the characterization of Theorem 8.7 in \cite{Billingsley1968}. This, together with tightness of the process  $\left\{X^{\epsilon}_{\cdot},\epsilon>0\right\}$, which is established in Theorem 3.2 of \cite{Spiliopoulos2012}, implies tightness of the pair $\left\{\left(X^{\epsilon}_{\cdot},\eta^{\epsilon}_{\cdot}\right),\epsilon>0\right\}$. Tightness of  $\left\{\eta^{\epsilon}_{\cdot},\epsilon>0\right\}$ in $C\left([0,T];\mathbb{R}^{m}\right)$  follows if we establish  that there is an $\epsilon_{0}>0$ such that for every $\eta>0$
\begin{enumerate}
\item{There exists $N<\infty$ such that
\begin{equation*}
\mathbb{P}\left[
\sup_{0\leq t\leq T}\left|\eta^{\epsilon}_{t}\right|>N\right] \leq \eta\quad\textrm{ for every }\epsilon\in(0,\epsilon_{0})
\end{equation*}}
\item{and, for every $M<\infty$
\begin{equation*}
\lim_{\rho\downarrow0}\sup_{\epsilon\in(0,\epsilon_{0})}\mathbb{P}\left[
\sup_{|t_{1}-t_{2}|<\rho,0\leq t_{1}<t_{2}\leq T}|\eta^\epsilon_{t_{1}}-\eta^\epsilon_{t_{2}}|\geq\eta, \sup_{0\leq t\leq T}|\eta^\epsilon_{t}|\leq M\right]  =0.
\end{equation*}
}
\end{enumerate}

By Duhamel's principle we can write
\begin{eqnarray}
\eta^\eps_{t}&=&\frac{\sqrt{\eps}}{\beta^{\epsilon}} \Theta^\eps_{x_0} (t)+ \frac{\frac{\epsilon}{\delta}-\gamma}{\beta^{\epsilon}} \Psi_{x_{0}} (t) \int_0^t \left[\Psi_{x_{0}} (s)\right]^{-1}  J\left( X_{s}^{\epsilon},Y_{s}^{\epsilon}\right) ds  \nonumber\\
& &\quad + \frac{1}{\beta^{\epsilon}}R^{\epsilon}(t;\Psi) +\Psi_{x_{0}} (t)\int_0^t \left[\Psi_{x_{0}} (s)\right]^{-1}\frac{1}{\beta^{\epsilon}}\Lambda[\bar{\lambda}]\parbar{ \bar X_s,X^\eps_s}ds,\label{eqn: Duhamel}
\end{eqnarray}
where

\begin{align} \notag
\Theta^{\epsilon}_{x_0} (t)&=    \Psi_{x_{0}} (t) \int_{0}^{t} \left[\Psi_{x_{0}} (s)\right]^{-1} \left(\delta\frac{\partial \Phi}{\partial x}\sigma+\sigma+\frac{\partial \Phi}{\partial y}\tau_{1}\right)\parbar{ X^\eps _s, Y^\eps _s } dW_{s}\nonumber\\
&\quad+  \Psi_{x_{0}} (t) \int_{0}^{t} \left[\Psi_{x_{0}} (s)\right]^{-1}\frac{\partial \Phi}{\partial y}\tau_{2}\parbar{ X^\eps _s, Y^\eps _s  }dB_s,\label{eqn: thetadef}
\end{align}
and
\begin{align*}
R^{\epsilon}(t;\Psi)&=-\delta\Psi_{x_{0}} (t)\left(\Phi \parbar{ X^\eps _t,  Y^\eps _t  }-\Phi \parbar{ X^\eps _0, Y^\eps _0  }\right)+\Psi_{x_{0}} (t)\int_{0}^{t}\left[\Psi_{x_{0}} (s)\right]^{-1}\mathcal{R}^\eps \parbar{ X^\eps _s, Y^\eps _s  }ds\nonumber
\end{align*}

The next step is to show that the third and the fourth term   on the right hand side of (\ref{eqn: Duhamel}) vanish in an appropriate way as $\epsilon\downarrow 0$.
To do so, notice that the process $\hat{Y}^{\epsilon}_{t}=Y^{\epsilon}_{\epsilon t}$ satisfies
\begin{eqnarray}
d\hat{Y}^{\epsilon}_{s}&=&\left[  \left(\frac{\epsilon}{\delta}\right)^{2}f\left(  X^{\epsilon}_{\epsilon s}%
,\hat{Y}^{\epsilon}_{s}\right)  + \frac{\epsilon}{\delta} g\left(  X^{\epsilon}_{\epsilon s}%
,\hat{Y}^{\epsilon}_{s}\right)\right] ds+\frac{\epsilon}{\delta}\left[
\tau_{1}\left(  X^{\epsilon}_{\epsilon s},\hat{Y}^{\epsilon}_{s}\right)
dW^{\epsilon}_{s}+\tau_{2}\left(  X^{\epsilon}_{\epsilon s},\hat{Y}^{\epsilon}_{s}\right)dB^{\epsilon}_{s}\right], \nonumber\\
\hat{Y}^{\epsilon}_{0}&=y_{0}\nonumber
\end{eqnarray}

where $W^{\epsilon}_{t}=W^{\epsilon}_{\epsilon t}$ and $B^{\epsilon}_{t}=B^{\epsilon}_{\epsilon t}$. This means that the law of $\hat{Y}^{\epsilon}_{s}$ is asymptotically identical to the law of a process corresponding to the operator $\gamma\mathcal{L}_{x}$. By Condition \ref{A:Assumption1} such a process has bounded moments. In particular, if $\bar{Y}_{t}(x)$ is the process corresponding to the operator $\gamma\mathcal{L}_{x}$, then Condition \ref{A:Assumption1} guarantees that
\[
\sup_{t\in[0,T]}\mathbb{E}_{y_{0}}|\bar{Y}_{t}(x)|^{q}\leq K(x)\left(1+|y_{0}|^{q}\right),
\]

where $K(x)$ is bounded with respect to $x$. By the definition of $\beta^{\epsilon}$, we  have that
\begin{eqnarray}
\left(\delta/\beta^{\epsilon}\right)\mathbb{E}\sup_{t\in[0,T]}\left[|\Phi\parbar{ X^\eps _t,  Y^\eps _t  }|+|\Phi \parbar{ x_0, y_0  }|\right]&\leq & \left(\delta/\beta^{\epsilon}\right) C \mathbb{E}\sup_{t\in[0,T]}\left[1+ |Y^\eps _t|^{q}\right]\nonumber\\
&\leq& \left(\delta/\beta^{\epsilon}\right) C \mathbb{E}\sup_{t\in[0,T]}\left[1+ |\hat{Y}^\eps_{t/\epsilon}|^{q}\right] \rightarrow 0, \textrm{ as } \epsilon\downarrow 0
\end{eqnarray}

The latter limit follows from the estimate (which is obtained analogously to Proposition 2 in \cite{PardouxVeretennikov1})
\[
\mathbb{E}_{y_{0}}\sup_{t\in[0,T]}\left|\hat{Y}^{\epsilon}_{t/\epsilon}\right|=o(1/\sqrt{\epsilon}) \textrm{ as }\epsilon\downarrow 0.
\]

Moreover, Theorem 2.5  guarantees that there is a $q$ such that
\begin{equation}
\left|\Phi(x,y)\right|+\left|\frac{\partial \Phi(x,y)}{\partial x}\right|+\left|\frac{\partial^{2} \Phi(x,y)}{\partial x^{2}}\right|+\left|\frac{\partial^{2} \Phi(x,y)}{\partial x\partial y}\right|\leq C \left(1+|y|^{q}\right)\label{Eq:Reg2DerivativesCondition}
\end{equation}

Consider now $\Xi(x,y)$ to be any of these functions
\begin{equation}
\textrm{i.e. } \Xi=\frac{\partial \Phi}{\partial x}b, \frac{\partial \Phi}{\partial x}c, \textrm{tr}\left(\frac{\partial^{2} \Phi}{\partial x^{2}}\sigma\sigma^{T}\right) \textrm{ or } \Xi=\textrm{tr}\left(\frac{\partial^{2} \Phi}{\partial x\partial y}\sigma\tau^{T}_{1}\right).\label{Eq:PsiTerms}
\end{equation}

Notice that these functions are the building blocks of  $\mathcal{R}^{\epsilon}(x,y)$ defined in Lemma \ref{L:Representationfor_I}. Let us define,
\begin{equation}
\theta^{\epsilon}(x,y)=\frac{\sqrt{\epsilon}}{\delta}\frac{\left|(\tau_{1}\tau_{1}^{T}(x,y)+\tau_{2}\tau_{2}^{T}(x,y))^{1/2}y\right|}{|y|}\label{Eq:TimeTrasnformation}
\end{equation}
and set $\varphi^{\epsilon}_{t}=\int_{0}^{t}\left|\theta^{\epsilon}\left(X^{\epsilon}_{s},Y^{\epsilon}_{s}\right)\right|^{2}ds$ and $\zeta^{\epsilon}_{t}=\left(\varphi^{\epsilon}_{t}\right)^{-1}$. If, we define
$\tilde{Y}^{\epsilon}_{t}=Y^{\epsilon}_{\zeta^{\epsilon}_{t}}$, then by Proposition $1$ in \cite{PardouxVeretennikov1}, we obtain that
\[
\mathbb{E}_{y_{0}}|\tilde{Y}^{\epsilon}_{t}|^{q}\leq C\left(1+|y_{0}|^{q}\right)
\]
for $\epsilon$ sufficiently small. Therefore, we have
\begin{eqnarray}
\mathbb{E}\sup_{t\in[0,T]}\int_{0}^{T}\left|\Xi\left(X^{\epsilon}_{s}, Y^{\epsilon}_{s}\right)\right|ds&\leq& C_{0}\mathbb{E}\int_{0}^{T}\left(1+\left|Y^{\epsilon}_{s}\right|^{q}\right)ds\nonumber\\
&\leq& C_{1} \mathbb{E}\int_{0}^{C_{1}T}\left(1+\left|\tilde{Y}^{\epsilon}_{s}\right|^{q}\right)ds\nonumber\\
&\leq& C_{2}T(1+|y_{0}|^{q})\nonumber
\end{eqnarray}
where $C_{i}$ are constants that depend on the bounds of the coefficients by Condition \ref{A:Assumption1}. The last computations, and the definition of $\beta^{\epsilon}$ imply then that
\begin{equation}
\lim_{\epsilon\downarrow 0}\mathbb{E}\sup_{t\in[0,T]}\left(\frac{1}{\beta^{\epsilon}}R^{\epsilon}(t;\Psi)\right)^{2}=0.\label{Eq:Reg2Term1goesToZero}
\end{equation}

Next we treat the fourth term in (\ref{eqn: Duhamel}). We want to prove that the process
\[
Q^{\epsilon}[\bar{\lambda};\Psi]_{t}=\Psi_{x_{0}} (t)\int_0^t \left[\Psi_{x_{0}} (s)\right]^{-1}\frac{1}{\beta^{\epsilon}}\Lambda[\bar{\lambda}]\parbar{ \bar X_s,X^\eps_s}ds
\]

converges to zero uniformly on $[0,T]$ in probability as $\epsilon\downarrow 0$.  
Let us define
\[
\tau^{\epsilon}=\inf\{ t>0: \left|X^{\epsilon}_{t}-\bar{X}_{t}\right|>\left|\beta^{\epsilon}\right|^{\rho}\}, \textrm{ for } \rho\in(1/2,1)
\]

The quadratic decay of $\Lambda[\bar{\lambda}]$ and $\rho>1/2$ imply that
\begin{equation}
\mathbb{E}\sup_{0\leq t\leq T\wedge \tau^{\epsilon}}|Q^{\eps}[\bar{\lambda};\Psi]_{t}|\rightarrow 0, \textrm{ as }\epsilon\downarrow 0.\label{Eq:Reg2ReminderTermGoesToZero}
\end{equation}

Hence it is enough to prove that $\lim_{\epsilon\downarrow 0}\mathbb{P}\left[\tau^{\epsilon}<T\right]=0$.  For this purpose, we notice that
for $\tau^\eps < T$ we have by (\ref{Eq:DeltaExpression}),
\begin{align*}
1 &= (\beta^\eps)^{1 - \rho }\sup_{0\leq t\leq T\wedge \tau^{\epsilon} } | \eta^\eps_{t} |\nonumber\\
&\leq (\beta^\eps)^{1 - \rho }\left[\sup_{0\leq t\leq T\wedge \tau^{\epsilon} } \left| \frac{\sqrt{\eps}}{\beta^{\epsilon}} \Theta^\eps_{x_0} (t) \right|+\sup_{0\leq t\leq T\wedge \tau^{\epsilon} } \left| \frac{\frac{\epsilon}{\delta}-\gamma}{\beta^{\epsilon}} \Psi_{x_{0}} (t) \int_0^t \left[\Psi_{x_{0}} (s)\right]^{-1}  J\left( X_{s}^{\epsilon},Y_{s}^{\epsilon}\right) ds \right|\right.\nonumber\\
&\qquad\qquad\left.+\sup_{0\leq t\leq T\wedge \tau^{\epsilon} } \left| \frac{1}{\beta^{\epsilon}}R^{\epsilon}(t;\Psi) \right|+\sup_{0\leq t\leq T\wedge \tau^{\epsilon} } \left| Q^{\eps}[\bar{\lambda};\Psi]_{t} \right|\right]\nonumber\\
&  = (\beta^\eps)^{1-\rho} C_1^{\epsilon},
\end{align*}
where $C_1^{\epsilon}$ is the random variable in the bracket. By the definition of $\beta^\eps$, tightness of $\Theta^\eps_{x_0} (t)$ and of $\Psi_{x_{0}} (t) \int_0^t \left[\Psi_{x_{0}} (s)\right]^{-1}  J\left( X_{s}^{\epsilon},Y_{s}^{\epsilon}\right) ds$, (\ref{Eq:Reg2Term1goesToZero}), (\ref{Eq:Reg2ReminderTermGoesToZero}) and because $\rho<1$, we obtain that the right hand side of the last display converges to zero in probability as $\epsilon,\delta\downarrow 0$. Hence, the claim $\lim_{\epsilon\downarrow 0}\mathbb{P}\left[\tau^{\epsilon}<T\right]=0$ follows. Therefore, we have shown 
\begin{equation}
\sup_{t\in[0,T]}\left|Q^{\eps}[\bar{\lambda};\Psi]_{t}\right|\rightarrow 0, \textrm{ in probability as }\epsilon\downarrow 0. \label{Eq:Reg2Term2goesToZero}
\end{equation}

Therefore, by (\ref{Eq:Reg2Term1goesToZero}) and (\ref{Eq:Reg2Term2goesToZero}) we have that the third and the fourth term of  (\ref{eqn: Duhamel}) converge to zero as $\epsilon\downarrow 0$.

Next it remains to consider the first and the second term on the right hand side of (\ref{eqn: Duhamel}). These terms do not vanish, but are bounded.

Let us first consider the first term on the right hand side of (\ref{eqn: Duhamel}), i.e., the term $\Theta^{\epsilon}_{x_{0}}(t)$.  By Doob's inequality for the martingale terms of $\Theta^{\epsilon}_{x_{0}}(t)$ and Theorem \ref{T:RegularityPoisson}, we have that

\begin{equation}
\mathbb{E}\sup_{t\in[0,T]}\left|\int_{0}^{t}\Xi_{s}\left(X^{\epsilon}_{s}, Y^{\epsilon}_{s}\right)dZ_{s}\right|^{2}\leq CT(1+|y_{0}|^{q})\label{Eq:XiStochasticIntegralTerms}
\end{equation}
where $Z_{\cdot}=W_{\cdot}$ or $B_{\cdot}$ and $\Xi=\left[\Psi_{x_{0}} (s)\right]^{-1} \sigma\frac{\partial \Phi}{\partial x}, \left[\Psi_{x_{0}} (s)\right]^{-1} \tau_{1}\frac{\partial \Phi}{\partial y}$ or $\left[\Psi_{x_{0}} (s)\right]^{-1} \tau_{2}\frac{\partial \Phi}{\partial y}$.

Similarly we can also bound the integrands of the second term on the right hand side of (\ref{eqn: Duhamel}). These estimates show that there exists $\epsilon_{0}>0$ small enough such that for every $\epsilon<\epsilon_{0}$

\[
\sup_{\epsilon\in(0,\epsilon_{0})}\mathbb{E}_{x_{0},y_{0}}\sup_{t\in[0,T]}|\eta^{\epsilon}_{t}|<\infty
\]
which implies part (i) of the requirements for tightness. In order to prove part (ii) of the requirements for tightness we define the random time
\[
\sigma^{\epsilon,M}=\inf\left\{  t\geq 0: |\eta^{\epsilon}_{t}|\geq M\right\}.
\]

So it suffices to show that for every $\eta$ and $M$ there exists $\epsilon_{0}$ and $\rho>0$ such that
\begin{equation*}
\sup_{\epsilon\in(0,\epsilon_{0})}\mathbb{P}\left[
\sup_{|t_{1}-t_{2}|<\rho,0\leq t_{1}<t_{2}\leq T\wedge \sigma^{\epsilon,M}}|\eta^\epsilon_{t_{1}}-\eta^\epsilon_{t_{2}}|\geq \eta\right] \leq \eta\rho.
\end{equation*}

This follows in a standard way by bounding the integrals that appear on the right side of the expression for $\eta^{\epsilon}_{t_{2}\wedge\sigma^{\epsilon,M}}-\eta^{\epsilon}_{t_{1}\wedge\sigma^{\epsilon,M}}$ based on (\ref{eqn: Duhamel}). In particular, by writing out $\eta^{\epsilon}_{t_{2}\wedge\sigma^{\epsilon,M}}-\eta^{\epsilon}_{t_{1}\wedge\sigma^{\epsilon,M}}$, we get an expression that involves integrals of the form $\int_{t_{1}\wedge \sigma^{\epsilon,M}}^{t_{2}\wedge \sigma^{\epsilon,M}}\left[\Psi_{x_{0}} (s)\right]^{-1}\Xi\left(X^{\epsilon}_{s}, Y^{\epsilon}_{s}\right) ds$, where $\Xi$ is any of the functions (\ref{Eq:PsiTerms}) and $J(x,y)$, and stochastic integrals of the form (\ref{Eq:XiStochasticIntegralTerms}). Using the change of time implied by (\ref{Eq:TimeTrasnformation}) and setting $\left(\tilde{X}^{\epsilon}_{t},\tilde{Y}^{\epsilon}_{t}\right)=\left(X^{\epsilon}_{\zeta^{\epsilon}_{t}},Y^{\epsilon}_{\zeta^{\epsilon}_{t}}\right)$, we obtain, similarly to the computations for part (i) of the tightness requirements, that if $\Xi(x,y)$ is any of the functions in (\ref{Eq:PsiTerms}), then
\begin{eqnarray}
\mathbb{E}\sup_{t\in[t_{1},t_{1}+\rho]}\int_{t_{1}\wedge \sigma^{\epsilon,M}}^{t\wedge \sigma^{\epsilon,M}}\left|\left[\Psi_{x_{0}} (s)\right]^{-1}\Xi\left(X^{\epsilon}_{s}, Y^{\epsilon}_{s}\right)\right|^{1+\nu} ds&\leq& C\rho^{\nu} \mathbb{E}\int_{t_{1}\wedge \sigma^{\epsilon,M}}^{(t_{1}+\rho)\wedge\sigma^{\epsilon,M}}\left(1+\left|\tilde{Y}^{\epsilon}_{s}\right|^{q(\nu)}\right)ds\nonumber\\
&\leq& C \rho^{1+\nu} (1+|y_{0}|^{q(\nu)})\nonumber
\end{eqnarray}
for sufficiently small $\nu>0$ and $q(\nu)$ a constant that depends on $q$ and $\nu$. Similar computations also hold for the stochastic integrals based on Doob's inequality. We omit the rest of the details.

From these considerations, tightness of the family $\{\eta^{\eps},\eps>0\}$ is being established.

\subsection{Identification of the limit}\label{SS:Reg2IdentificationLimit}

We identify the limit using the martingale problem formulation. For this purpose we apply It\^{o} formula to a function $\phi\in C^{2}_{b}(\mathbb{R}^{m})$ with process

\[
\psi^{\epsilon}_{t}=\eta^{\epsilon}_{t}+\left(\delta/\beta^{\epsilon}\right)\left(\Phi \parbar{ X^\eps _t,  Y^\eps _t  }-\Phi \parbar{ X^\eps _0, Y^\eps _0  }\right)
\]
We get

\begin{align}
\phi(\psi^\epsilon_{t}) &=   \int_{0}^{t} D_x \bar{\lambda} ( \bar X_s) \eta^\eps_{s} D \phi(\psi^{\eps}_{s}) ds+
\frac{\left(\frac{\epsilon}{\delta}-\gamma\right)}{\beta^{\epsilon}}\int_{0}^{t}J \left( X_{s}^{\epsilon},Y_{s}^{\epsilon}\right)  D \phi(\psi^{\epsilon}_{s}) ds\nonumber\\
&\hspace{0.3cm}+\int_{0}^{t}\frac{1}{\beta^{\epsilon}}\left[\mathcal{R}^{\epsilon}\parbar{ X^\eps _s,Y_{s}^{\epsilon}  }+\Lambda[\bar{\lambda}]\parbar{ \bar X_s,X^\eps_s}\right]D \phi(\psi^{\epsilon}_{s})ds+\nonumber\\
&\hspace{0.3cm}+\left(\frac{\sqrt{\epsilon}}{\beta^{\epsilon}}\right)^{2}\frac{1}{2}
\int_{0}^{t}\textrm{tr}\left[ D^{2} \phi(\psi^{\epsilon}_{s}) \left(\delta\frac{\partial \Phi}{\partial x}\sigma+\sigma+\frac{\partial \Phi}{\partial y}\tau_{1}\right)\left(\delta\frac{\partial \Phi}{\partial x}\sigma+\sigma+\frac{\partial \Phi}{\partial y}\tau_{1}\right)^{T} \right]\parbar{ X^\eps _s,Y_{s}^{\epsilon}  }ds\nonumber\\
&\hspace{0.3cm}+\left(\frac{\sqrt{\epsilon}}{\beta^{\epsilon}}\right)^{2}\frac{1}{2}
\int_{0}^{t}\textrm{tr}\left[ D^{2} \phi(\psi^{\epsilon}_{s}) \left(\frac{\partial \Phi}{\partial y}\tau_{2}\right)\left(\frac{\partial \Phi}{\partial y}\tau_{2}\right)^{T} \right]\parbar{ X^\eps _s,Y_{s}^{\epsilon}  }ds\nonumber\\
&\hspace{0.3cm}+\frac{\sqrt{\epsilon}}{\beta^{\epsilon}}\int_{0}^{t} D \phi(\psi^{\epsilon}_{s}) \left(\delta\frac{\partial \Phi}{\partial x}\sigma+\sigma+\frac{\partial \Phi}{\partial y}\tau_{1}\right)\parbar{ X^\eps _s,Y_{s}^{\epsilon}  }dW_s\nonumber\\
&\hspace{0.3cm}+\frac{\sqrt{\epsilon}}{\beta^{\epsilon}}\int_{0}^{t} D \phi(\psi^{\epsilon}_{s}) \left(\frac{\partial \Phi}{\partial y}\tau_{2}\right)\parbar{ X^\eps _s,Y_{s}^{\epsilon}  }dB_s
\label{Eq:PhiExpression}
\end{align}

We have two cases to consider, depending on whether $\ell\neq 0$ or $\ell=0$.

Let us first assume that $\ell=\lim_{\epsilon\downarrow 0}\frac{\sqrt{\epsilon}}{\left(\frac{\epsilon}{\delta}-\gamma\right)}\neq 0$. In this case $\beta^{\epsilon}=\sqrt{\epsilon}$ and the result follows if we prove that for any $0\leq s\leq t\leq T$
\begin{eqnarray}
& &\lim_{\epsilon\downarrow 0}\mathbb{E}\left[\phi(\eta^{\epsilon}_{t})-\phi(\eta^{\epsilon}_{s})-\int_{s}^{t}\left[\left(D_x \bar{\lambda} ( \bar X_r) \eta^\eps_{r}+\ell^{-1}\bar{J} \left( X_{r}^{\epsilon}\right)\right)D \phi(\eta^{\eps}_{r})\right.\right.\nonumber\\
& &\hspace{5.3cm}\left.\left.+ \frac{1}{2}\textrm{tr}\left[ D^{2} \phi(\eta^{\epsilon}_{r}) q\left( X_{r}^{\epsilon}\right) \right]\right]dr\Big | \mathcal{F}_{s}\right]=0\label{Eq:Martingale}
\end{eqnarray}

For this purpose, we first notice that, as in the proof of tightness,
\begin{equation}
\left(\delta/\beta^{\epsilon}\right)\mathbb{E}\sup_{t\in[0,T]}\left[|\Phi\parbar{ X^\eps _t,  Y^\eps _t  }|+|\Phi \parbar{ x_0, y_0  }|\right]\rightarrow 0, \textrm{ as } \epsilon\downarrow 0\label{Eq:XiTerm1}
\end{equation}
and
\begin{equation*}
\lim_{\epsilon\downarrow 0}\mathbb{E}\left[\sup_{t\in[0,T]} \int_{0}^{t}\frac{1}{\beta^{\epsilon}}\left|\mathcal{R}^{\epsilon}\parbar{ X^\eps _s,Y_{s}^{\epsilon}  }D \phi(\psi^{\epsilon}_{s})\right|ds+ \int_{0}^{t}\frac{1}{\beta^{\epsilon}}\Lambda[\bar{\lambda}]\parbar{ \bar X_s,X^\eps_s}D \phi(\psi^{\epsilon}_{s})ds\right]=0
\end{equation*}

Moreover, the stochastic integrals in (\ref{Eq:PhiExpression}) are square integrable. This follows from Doob's inequality and Theorem \ref{T:RegularityPoisson}. Thus, their expected value vanishes in the prelimit.

Next notice that by construction $\frac{\left(\frac{\epsilon}{\delta}-\gamma\right)}{\sqrt{\epsilon}}\rightarrow \ell^{-1}$. So, thanks to (\ref{eqn: Duhamel}),(\ref{Eq:PhiExpression}) and (\ref{Eq:XiTerm1}), it essentially remains to prove that
\begin{equation}
\lim_{\epsilon\downarrow 0}\mathbb{E}\left[\int_{s}^{t}J \left( X_{r}^{\epsilon},Y_{r}^{\epsilon}\right)  D \phi(\eta^{\epsilon}_{r}) dr-\int_{s}^{t}\bar{J} \left( X_{r}^{\epsilon}\right)D \phi(\eta^{\eps}_{r})dr\Big | \mathcal{F}_{s}\right]=0\label{Eq:Regime2LimitIdentTerm1}
\end{equation}
and
\begin{eqnarray}
& &\lim_{\epsilon\downarrow 0}\mathbb{E}\left[
\int_{s}^{t}\textrm{tr}\left[ D^{2} \phi(\eta^{\epsilon}_{r}) \left(\sigma+\frac{\partial \Phi}{\partial y}\tau_{1}\right)\left(\sigma+\frac{\partial \Phi}{\partial y}\tau_{1}\right)^{T}+ D^{2} \phi(\eta^{\epsilon}_{r}) \left(\frac{\partial \Phi}{\partial y}\tau_{2}\right)\left(\frac{\partial \Phi}{\partial y}\tau_{2}\right)^{T}  \right]\parbar{ X^\eps _r,Y_{r}^{\epsilon}  }dr\right.\nonumber\\
& &\hspace{9cm} \left.
-\int_{s}^{t}\textrm{tr}\left[ D^{2} \phi(\eta^{\epsilon}_{r}) q\left( X_{r}^{\epsilon}\right) \right]dr\Big | \mathcal{F}_{s}\right]=0\label{Eq:Regime2LimitIdentTerm2}
\end{eqnarray}

Due to tightness of the pair $\{(X^{\epsilon},\eta^{\epsilon}),\epsilon>0\}$ there is a subsequence that converges weakly to a process $(\bar{X},\eta)$. To prove that (\ref{Eq:Regime2LimitIdentTerm1}) and (\ref{Eq:Regime2LimitIdentTerm2}) hold we use the standard idea of freezing the slow component $X^{\eps}_{\cdot}$, see for example \cite{FWBook,PardouxVeretennikov2}, and the ergodic theorem. The details are omitted. This concludes the proof for the case $\ell\neq 0$.

We finally consider the case $\ell=0.$ Here the limiting process $\bar{\eta}_{\cdot}$ is deterministic. Convergence will follow if we prove that

\begin{equation*}
 \lim_{\epsilon\downarrow 0}\mathbb{E}\left[\phi(\eta^{\epsilon}_{t})-\phi(\eta^{\epsilon}_{s})-\int_{s}^{t}\left[\left(D_x \bar{\lambda} ( \bar X_r) \eta^\eps_{r}+\bar{J} \left( X_{r}^{\epsilon}\right)\right)D \phi(\eta^{\eps}_{r})\right]dr\Big | \mathcal{F}_{s}\right]=0\label{Eq:Martingale2}
\end{equation*}
This follows by arguments very similar to those of the previous case with $\ell\neq 0$.

\section{Proof of Theorem \ref{T:CLT2} for Regime 1.}\label{S:ProofRegime2}
In this section we consider Regime 1, i.e. we assume that $\epsilon/\delta\rightarrow \infty$ as $\epsilon\downarrow 0$. As in Regime 2, we omit the subscript $1$ from the functions $\lambda,J,q$ and measure $\mu$. The situation here is more complex than in Regime 2, due to the unclear behavior of the integral term $\frac{\epsilon}{\delta}\int_{0}^{t}b\left(  X^{\epsilon}_{s}%
,Y^{\epsilon}_{s}\right) ds$. To go around this we consider a function $\chi=(\chi_{1},\ldots,\chi_{m})$, which grows at most polynomially in $y$ as $|y|\rightarrow \infty$,  and satisfies the Poisson equation
\begin{equation}
\mathcal{L}_{x}^{1}\chi_{l}(x,y)=-b_{l}(x,y),\quad\int_{\mathcal{Y}}%
\chi_{l}(x,y)\mu(dy|x)=0,\hspace{0.1cm} l=1,...,m.\nonumber
\end{equation}
By applying  It\^{o}-Krylov's formula to $\chi(x,y)=(\chi_{1}(x,y
),\ldots,\chi_{m}(x,y))$ with $(x,y)=(X_{t}^{\epsilon},Y_{t}^{\epsilon})$, we can reduce the problem to the previous case.  Note that by Condition \ref{A:Assumption1}, Theorem \ref{T:RegularityPoisson} applies and thus $\chi$ has the required regularity. By doing so, we can rewrite the first component of (\ref{Eq:Main}), omitting function arguments in some places for notational convenience,  as

\begin{eqnarray}
X^{\epsilon}_{t}&=&  x_{0}+\int_{0}^{t}\lambda\left(  X^{\epsilon}_{s},Y^{\epsilon}_{s}\right) ds- \delta \left( \chi \parbar{ X^\eps _t, Y^{\epsilon}_{t}  }- \chi \parbar{ x_{0}, y_{0} }\right)\nonumber\\
& &\quad +\int_{0}^{t}\left(\epsilon \frac{\partial \chi}{\partial x}b+\delta\frac{\partial \chi}{\partial x}c+\frac{\epsilon\delta}{2}\textrm{tr}\left[\sigma\sigma^{T}\frac{\partial^{2} \chi}{\partial x^{2}}\right]+\epsilon\textrm{tr}\left[\sigma\tau_{1}^{T}\frac{\partial^{2} \chi}{\partial x\partial y}\right]\right)\left(  X^{\epsilon}_{s},Y^{\epsilon}_{s}\right)   ds\nonumber\\
& &\quad+\sqrt{\epsilon}\int_{0}^{t}
\left(\sigma+\frac{\partial \chi}{\partial y}\tau_{1}+\delta\frac{\partial \chi}{\partial x}\sigma \right)\left(  X^{\epsilon}_{s},Y^{\epsilon}_{s}\right)
dW_{s}+\sqrt{\epsilon}\int_{0}^{t}\frac{\partial \chi}{\partial y}\tau_{2}\left(  X^{\epsilon}_{s},Y^{\epsilon}_{s}\right)dB_{s}\nonumber
\end{eqnarray}

Then, as in the case of Regime $2$, we obtain that $\Delta^{\epsilon}_{t}=X^{\epsilon}_{t}-\bar{X}_{t}$ satisfies

\begin{align*}
 \Delta^{\epsilon}_{t} &= \int_{0}^{t}D_x \bar{\lambda} ( \bar X_s) \Delta^\eps_{s}ds
+ \int_{0}^{t}\parbar{ \lambda \parbar{ X^\eps _s, Y^{\epsilon}_{s} }  - \bar{\lambda} \left( X^\eps _s\right)   }ds
 \\
& \quad +\sqrt{\epsilon} \int_{0}^{t}\left(\sigma+\frac{\partial \chi}{\partial y} \tau_{1}+\delta \frac{\partial \chi}{\partial x} \sigma \right)\parbar{ X^\eps _s, Y^{\epsilon}_{s} } dW_{s} +
\sqrt{\epsilon}\int_{0}^{t}\frac{\partial \chi}{\partial y}\tau_{2}\left(  X^{\epsilon}_{s},Y^{\epsilon}_{s}\right)dB_{s}\\
&\quad+\int_{0}^{t}\mathcal{R}^\eps_{1}\parbar{ X^\eps _s, Y^{\epsilon}_{s} }ds
+\int_{0}^{t}\Lambda[{\lambda}_{1}]\parbar{ \bar X_s,X^\eps_s}dt- \delta \left( \chi \parbar{ X^\eps _t, Y^{\epsilon}_{t}  }- \chi \parbar{ x_{0}, y_{0} }\right)
\end{align*}
where

\[
\mathcal{R}^\eps_{1}\parbar{x,y}=\left(\epsilon \frac{\partial \chi}{\partial x}b+\delta\frac{\partial \chi}{\partial x}c+\frac{\epsilon\delta}{2}\textrm{tr}\left[\sigma\sigma^{T}\frac{\partial^{2} \chi}{\partial x^{2}}\right]+\epsilon\textrm{tr}\left[\sigma\tau_{1}^{T}\frac{\partial^{2} \chi}{\partial x\partial y}\right]\right)\left(x,y\right)
\]

Next, we need to understand the behavior of the correction term

\[
\Gamma^{\epsilon,\delta}_{t}=\int_{0}^{t}\parbar{ \lambda \parbar{ X^\eps _s, Y^{\epsilon}_{s} }  - \bar{\lambda} \left( X^\eps _s\right)   }ds
\]

Recall that $\Phi$ satisfies the Poisson equation (\ref{Eq:CellProblemCLT}) with $\lambda=\lambda_{1}$. We have the following lemma, which is exactly analogous to Lemma \ref{L:Representationfor_I} for Regime $2$.

\begin{lemma}\label{L:Representationfor_II}
Assume Conditions \ref{A:Assumption1} and \ref{A:LyapunovCondition}. The following hold
\begin{enumerate}
\item{The solution to the Poisson equation (\ref{Eq:CellProblemCLT}) satisfies the conclusions of Theorem \ref{T:RegularityPoisson}.}
\item{For every $\epsilon,\delta>0$ we have the representation

\begin{align*}
\Gamma^{\epsilon,\delta}_{t}&=\int_{0}^{t}\left[\lambda  \left( X_{s}^{\epsilon},Y_{s}^{\epsilon}\right) - \bar{\lambda}(X_{s}^{\epsilon})\right]  ds\nonumber\\
&=
\frac{\delta}{\epsilon}\int_{0}^{t}J\left( X_{s}^{\epsilon},Y_{s}^{\epsilon}\right)  ds-\frac{\delta^{2}}{\epsilon}\left(\Phi\parbar{ X^\eps _t, Y_{s}^{\epsilon}  }-\Phi\parbar{ X^{\epsilon}_0, Y^{\epsilon}_{0}  }\right)+\int_{0}^{t}\mathcal{R}^{\epsilon}_{2}\left( X_{s}^{\epsilon},Y_{s}^{\epsilon}\right)ds\nonumber\\
&\hspace{0.3cm}
+\frac{\delta}{\sqrt{\epsilon}}\int_{0}^{t} \left(\delta\frac{\partial \Phi}{\partial x}\sigma+\frac{\partial \Phi}{\partial y}\tau_{1}\right)\parbar{ X^\eps _s, Y^\eps _s }dW_s+\frac{\delta}{\sqrt{\epsilon}}\int_{0}^{t} \frac{\partial \Phi}{\partial y}\tau_{2}\parbar{ X^\eps _s, Y^\eps _s  }dB_s\nonumber
\end{align*}
where
\[
\mathcal{R}^{\epsilon}_{2}(x,y)=\left[\delta \frac{\partial \Phi}{\partial x}b+\frac{\delta^{2}}{\epsilon} \frac{\partial \Phi}{\partial x}c+\frac{\delta^{2}}{2}\textrm{tr}\left(\sigma\sigma^{T}\frac{\partial^{2} \Phi}{\partial x^{2}}\right)
+\delta\textrm{tr}\left(\sigma\tau_{1}^{T}\frac{\partial^{2} \Phi}{\partial x\partial y}\right)\right](x,y).
\]}
\end{enumerate}
\end{lemma}

\begin{proof}
Part (i). We need to verify that the right hand side of (\ref{Eq:CellProblemCLT}), i.e., $G(x,y)=\lambda(x,y)-\bar{\lambda}(x)$ satisfies the assumptions of Theorem \ref{T:RegularityPoisson}. Keeping in mind that in Regime 2, we have $\lambda(x,y)=c(x,y)+\frac{\partial \chi}{\partial y}g(x,y)$, the smoothness and growth conditions are satisfied for $c(x,y)$ and $g(x,y)$ due to Condition \ref{A:Assumption1}. For the corrector term $\frac{\partial \chi}{\partial y}$ we need the estimate
\[
\left|\frac{\partial^{3} \chi}{\partial y\partial x^{2}}(x,y)\right|\leq K\left(1+|y|^{q}\right)
\]
This is not immediately implied by Theorem \ref{T:RegularityPoisson}, but due to Condition \ref{A:Assumption1} is true via Theorem 1 in \cite{PardouxVeretennikov1}.

For $\bar{\lambda}(x)=\int_{\mathcal{Y}}\lambda(x,y)\mu(dy|x)$ the same is true if the invariant measure $\mu(dy|x)$ is appropriately smooth. This follows from the estimates in Theorem 1 of \cite{PardouxVeretennikov2}.

Part (ii). By part (i) we can apply the  It\^{o}-Krylov formula for functions with Sobolev derivatives to $\Phi(x,y)=\left(\Phi_{1}(x,y),\cdots,\Phi_{m}(x,y)\right)$ with $(x,y)=\left(X^{\epsilon}_{t},Y^{\epsilon}_{t}\right)$. The rest follow as in Lemma \ref{L:Representationfor_I} and thus the details are omitted.
\end{proof}

Hence, by Lemma \ref{L:Representationfor_II} we get that $\eta^{\epsilon}_{t}=\Delta^\epsilon_{t}/\beta^{\epsilon}$ satisfies

\begin{align}
\eta^\epsilon_{t} &=  \int_{0}^{t} D_x \bar{\lambda} ( \bar X_s) \eta^\eps_{s} ds+
\frac{\delta/\epsilon}{\beta^{\epsilon}}\int_{0}^{t}J  \left( X_{s}^{\epsilon},Y_{s}^{\epsilon}\right)  ds\label{Eq:DeltaExpressionReg2}\\
&\hspace{0.3cm}+\int_{0}^{t}\frac{1}{\beta^{\epsilon}}\left[\mathcal{R}^{\epsilon}_{1}+\mathcal{R}^{\epsilon}_{2}\right]\parbar{ X^\eps _s, Y^\eps _s }ds+ \int_0^t \frac{1}{\beta^{\epsilon}}\Lambda[\bar{\lambda}]\parbar{ \bar X_s,X^\eps_s}ds\nonumber\\
&\hspace{0.3cm}-\frac{\left(\delta^{2}/\epsilon\right)}{\beta^{\epsilon}}\left(\Phi \parbar{ X^\eps _t, Y^\eps _t   }-\Phi \parbar{ X^\eps _0, Y^\eps _0 }\right)- \left(\delta/\beta^{\epsilon}\right) \left( \chi \parbar{ X^\eps _t, Y^{\epsilon}_{t}  }- \chi \parbar{ x_{0}, y_{0} }\right)\nonumber\\
&\hspace{0.3cm}+\frac{\delta/\sqrt{\epsilon}}{\beta^{\epsilon}}\int_{0}^{t} \left(\delta\frac{\partial \Phi}{\partial x}\sigma+\sigma+\frac{\partial \Phi}{\partial y}\tau_{1}\right)\parbar{ X^\eps _s, Y^\eps _s }dW_s+\frac{\delta/\sqrt{\epsilon}}{\beta^{\epsilon}}\int_{0}^{t} \frac{\partial \Phi}{\partial y}\tau_{2}\parbar{ X^\eps _s, Y^\eps _s  }dB_s\nonumber\\
& \quad +\frac{\sqrt{\epsilon}}{\beta^{\epsilon}} \int_{0}^{t}\left(\sigma+\frac{\partial \chi}{\partial y} \tau_{1}+\delta \frac{\partial \chi}{\partial x} \sigma \right)\parbar{ X^\eps _s, Y^{\epsilon}_{s} } dW_{s}
+\frac{\sqrt{\epsilon}}{\beta^{\epsilon}}\int_{0}^{t}\frac{\partial \chi}{\partial y}\tau_{2}\left(  X^{\epsilon}_{s},Y^{\epsilon}_{s}\right)dB_{s}\nonumber
  \end{align}

Then, the proof follows the same steps as in the case of Regime $2$, except for two minor modifications, which we now explain, even though we will not repeat the proof. The first modification is that in the tightness proof, the corresponding process $\hat{Y}$ is defined to be $\hat{Y}^{\epsilon}_{t}=Y^{\epsilon}_{\frac{\delta^{2}}{\epsilon}t}$. The second modification is that for the identification of the limit, we apply It\^{o} formula to a smooth function with the process

\begin{equation*}
\psi^{\epsilon}_{t}=\eta^{\epsilon}_{t}+\left(\left(\delta^{2}/\epsilon\right)/\beta^{\epsilon}\right)\left(\Phi \parbar{ X^\eps _t,  Y^\eps _t }-\Phi \parbar{ x_0, y_0 }  \right)+\left(\delta/\beta^{\epsilon}\right) \left(\chi \parbar{ X^\eps _t, Y^\eps _t }-\chi \parbar{ x_0, y_0  }\right).
\end{equation*}

\section{An Example and Connections to Large Deviations}\label{S:Examples}
In this section we present an example to illustrate our results. We consider a stochastic model with two scales, one faster than the other one:

\begin{eqnarray}
dX_{s}&=&c^{\epsilon}\left(X_{s},Y_{s}\right)ds+\sigma\left(Y_{s}\right)
dW_{s}, \label{Eq:SVModela}\\
dY_{s}&=& \frac{1}{\delta^{2}}\left(m-Y_{s}\right)   ds+\frac{1}{\delta}\left[
\rho dW_{s}+\sqrt{1-\rho^{2}}dB_{s}\right]\nonumber
\end{eqnarray}
where $0<\epsilon,\delta\ll 1$, $m\in\mathbb{R}$ and $\rho\in[-1,1]$ is the correlation between the noise of the $X$ and $Y$ process. Assume that $c^{\epsilon}(x,y)$ is such that
\[
\epsilon c^{\epsilon}(x,y)\rightarrow c(x,y)\quad \textrm{ uniformly in }(x,y) \textrm{ for some }c(x,y)\textrm{ as }\epsilon\downarrow 0.
\]

Assume that $c(x,y)$ and $\sigma(y)$ satisfy Condition \ref{A:Assumption1}. If, we are interested in short time asymptoptics, then it is convenient to change time $s\mapsto \epsilon s$ with $0<\epsilon\ll 1$.  Writing the system under the new timescale, we obtain $\left\{\left(X^{\epsilon}_{s}, Y^{\epsilon}_{s}\right), s\in[0,T]\right\}$ as the unique strong solution to:
\begin{eqnarray}
dX^{\epsilon}_{s}&=&\epsilon c^{\epsilon}\left(X^{\epsilon}_{s}, Y^{\epsilon}_{s}\right)ds+\sqrt{\epsilon}%
\sigma\left(Y^{\epsilon}_{s}\right)
dW_{s}, \label{Eq:SVModel}\\
dY^{\epsilon}_{s}&=& \frac{\epsilon}{\delta^{2}}\left(m-Y^{\epsilon}_{s}\right)   ds+\frac{\sqrt{\epsilon}}{\delta}\left[
\rho dW_{s}+\sqrt{1-\rho^{2}}dB_{s}\right]\nonumber
\end{eqnarray}

Both components $(X,Y)$ take values in $\mathbb{R}$. So, we have $\mathcal{Y}=\mathbb{R}$. We supplement the system with initial condition $(X^{\epsilon}(0),Y^{\epsilon}(0))=(x_{0},y_{0})$. To connect to the notation of the general model (\ref{Eq:Main}), this corresponds to
\[
b(x,y)=0,~~~\sigma(x,y)=\sigma(y),
 \]
 \[f(x,y)=m-y,~~~ g(x,y)=0, ~~~\tau_{1}(x,y)=\rho, ~~~\tau_{2}(x,y)=\sqrt{1-\rho^{2}}.
\]

\subsection{Law of large numbers and central limit theorem}

Let us see how our theorem applies for this model. For both regimes, the invariant measure corresponding to the fast motion is actually the same, it is Gaussian, does not depend on $x$ and it is equal to
\begin{equation*}
\mu(dy)=\frac{1}{\sqrt{\pi}} e^{-(y-m)^{2}}dy
\end{equation*}

The law of large numbers is given in the following corollary.
\begin{corollary}
Consider Regime $i=1,2$. Under the assumptions of Theorem \ref{T:LLN}, we have that $X^{\epsilon}_{t}$ converges in probability, uniformly in $t\in[0,T]$ as $\epsilon\downarrow 0$ to the deterministic function $\bar{X}^{i}_{t}(x_{0})$. For both $i=1$ and $i=2$ we have that $\bar{X}_{t}(x_{0})=\bar{X}^{1}_{t}(x_{0})=\bar{X}^{2}_{t}(x_{0})$ satisfies the ODE
\[
\bar{X}_{t}=x_{0}+\int_{0}^{t} \left[\int_{\mathcal{Y}}c(\bar{X}_{s},y)\frac{1}{\sqrt{\pi}} e^{-(y-m)^{2}}dy \right]ds
\]
\end{corollary}

\begin{proof}
It follows trivially by the Definition \ref{Def:ThreePossibleFunctions} and Theorem \ref{T:LLN}.
\end{proof}

Thus, the law of large numbers is the same independently of the order that $\epsilon$ and $\delta$ go to zero. But this is not the same for the second order correction. In particular, let us  define for $F\in C^{2}(\mathcal{Y})$ the operator
\[
\mathcal{L}F(y)=(m-y)F'(y)+\frac{1}{2}F''(y)
\]

Then, it is easy to see that $\Phi_{2}$ is given by

\begin{eqnarray}
& & \mathcal{L}\Phi_{2}(x,y)=-\frac{1}{\gamma}\left(c\left( x,y\right) - \bar{c}(x)\right),\quad\int_{\mathcal{Y}}%
\Phi_{2}(x,y)\mu(dy)=0,\hspace{0.1cm} \label{Eq:PoissonEquationExample}\\
& & \Phi_{2} \textrm{ grows at most polynomially in }y\textrm{ as } |y|\rightarrow\infty\nonumber
\end{eqnarray}
where $\bar{c}(x)=\frac{1}{\sqrt{\pi}}\int_{\mathcal{Y}}c(x,y) e^{-(y-m)^{2}}dy$. Due to our assumptions, Theorem \ref{T:RegularityPoisson} applies to $\Phi_{2}(x,y)$. Moreover, we have
\begin{eqnarray}
J_{1}(x,y)&=&0\nonumber\\
q_{1}(x,y)&=&\sigma^{2}(y)\nonumber
\end{eqnarray}
and
\begin{eqnarray}
J_{2}(x,y)&=&\mathcal{L}\Phi_{2}(x,y)=-\frac{1}{\gamma}\left(c(x,y)-\bar{c}(x)\right).\nonumber\\
q_{2}(x,y)&=&\sigma^{2}(y)+\left(\partial_{y}\Phi_{2}(x,y)\right)^{2}+2\rho\sigma(y)\partial_{y}\Phi_{2}(x,y)\nonumber
\end{eqnarray}

Recall that for any function $f(x,y)$ we denote by $\bar{f}(x)=\int_{\mathcal{Y}}f(x,y)\mu(dy)$.  So, by construction we get that $\bar{J}_{2}(x)=0$. Moreover, if we denote by
\[
q=\frac{1}{\sqrt{\pi}}\int_{\mathcal{Y}}\sigma^{2}(y) e^{-(y-m)^{2}}dy
\]
then,
\[
\bar{q}_{1}(x)=q,\quad \textrm{and } \quad  \bar{q}_{2}(x)=q+\frac{1}{\sqrt{\pi}}\int_{\mathcal{Y}}\left[\left(\partial_{y}\Phi_{2}(x,y)\right)^{2}+2\rho\sigma(y)\partial_{y}\Phi_{2}(x,y)\right] e^{-(y-m)^{2}}dy
\]
Theorem \ref{T:CLT2} translates to the following corollary.

\begin{corollary}\label{C:CLT2}
Under the notation and assumptions of Theorem \ref{T:CLT2} we have that the process
$$\eta^{\epsilon}_{t}=\frac{X^{\epsilon}_{t}-\bar{X}_{t}}{\beta^{\epsilon}_{i}}$$ converges weakly in the space of continuous functions in $\mathcal{C}\left([0,T];\mathbb{R}\right)$
to the solution of the Ornstein-Uhlenbeck type of process

\begin{enumerate}
\item{Regime 1:
\begin{eqnarray}
 d\eta_{t}&=&D\bar{c}(\bar{X}_{t}(x_{0}))\eta_{t}dt+\one \left(\ell_{1}\neq 0 \right)\bar{q}^{1/2}_{1}(\bar{X}_{t}(x_{0}))d\tilde{W}_{t}\nonumber\\
\eta_{0}&=&0.\nonumber
\end{eqnarray}
\item{Regime 2:
\begin{eqnarray}
 d\eta_{t}&=&D\bar{c}(\bar{X}_{t}(x_{0}))\eta_{t}dt +\one \left(\ell_{2}\neq 0 \right)\bar{q}^{1/2}_{2}(\bar{X}_{t}(x_{0}))d\tilde{W}_{t}\nonumber\\
\eta_{0}&=&0,\nonumber
\end{eqnarray}}
}
\end{enumerate}
where $\tilde{W}$ is a standard Wiener process.
\end{corollary}

Therefore, even though the law of large numbers limit happens to be the same independently of the order that $\epsilon$ and $\delta$ go to zero, the situation changes when one considers the second order correction given by the fluctuation analysis. In particular, the two limiting processes arising from the fluctuations have different diffusion coefficients. For example in the case $\rho=0$ and if $c(x,y)$ is not identically zero or only a function of $x$, then it is easy to see that $\bar{q}_{2}(x)>\bar{q}_{1}(x)$.

As a specific example, let us consider the simple case $c(x,y)=y^{2}$. Then, it can be easily verified that $\bar{c}=\frac{1}{2}+m^{2}$. Moreover, by direct substitution we can check that
\[
\Phi_{2}(y)=\frac{\frac{1}{2}y^{2}+my-\frac{1}{4}-\frac{3}{2}m^{2}}{\gamma}
\]

satisfies the Poisson equation (\ref{Eq:PoissonEquationExample}). Thus a straightforward computation shows that
\[
\bar{q}_{2}=q+\frac{4m^{2}+\frac{1}{2}}{\gamma^{2}}+\frac{2\rho}{\gamma\sqrt{\pi}}\int_{\mathbb{R}}\sigma(y)(y+m)e^{-(y-m)^{2}}dy
\]
Notice that if $\bar{q}_{2}$ is viewed as function of $\gamma$, then $\bar{q}_{2}(\gamma)\rightarrow \bar{q}_{1}$ as $\gamma\rightarrow\infty$ and that in the case $\rho=0$
\[
\bar{q}_{2}=q+\frac{4m^{2}+\frac{1}{2}}{\gamma^{2}}>q=\bar{q}_{1}.
\]

\subsection{Connections to Large Deviations}\label{S:LDP}
In \cite{Spiliopoulos2012}, sample path large deviations principle for the family $\left\{X^{\epsilon}_{\cdot},\epsilon\in(0,1)\right\}$ is established in the case where the coefficients of the system (\ref{Eq:Main}) are periodic in the fast motion $y$. In general, it is known that central limit theorems are related to the second derivative of the related rate functions. A standard example is the case of Cramer's theorem, see \cite{Hollander2000}. Let us outline this connection in our setup in the simple case of the example (\ref{Eq:SVModela}) assuming that
$\lim_{\epsilon\downarrow 0}\epsilon c^{\epsilon}(x,y)=0$ uniformly in $x$ and $y$ in the case of Regime 1. The discussion that follows is heuristic, but illustrative of the connection between central limit theorems and large deviations for multiple scale diffusion processes.

If one is interested in short time asymptotics, then we can either consider the random variables $X_{t}$ satisfying (\ref{Eq:SVModela}) as $t\downarrow 0$, or equivalently $X^{\epsilon}_{1}$ satisfying (\ref{Eq:SVModel}) as $\epsilon\downarrow 0$ for $t=1$.  By  Theorem 3.4 in \cite{Spiliopoulos2012}, we have that the action functional for the random variables $X_{t}$ should satisfy a large deviations principle as $t\downarrow 0$ with rate function given by
\begin{eqnarray}
S(x_{1})&=&\frac{1}{2}\inf_{\phi\in \mathcal{AC}([0,1];\mathbb{R}), \phi_{0}=x_{0},\phi_{1}=x_{1}}\int_{0}^{1}\frac{|\dot{\phi}_{s}|^{2}}{q}ds\nonumber\\
&=&\frac{1}{2}\inf_{\phi\in \mathcal{AC}([0,1];\mathbb{R}), \phi_{0}=x_{0}/\sqrt{q},\phi_{1}=x_{1}/\sqrt{q}}\int_{0}^{1}|\dot{\phi}_{s}|^{2}ds,
\end{eqnarray}
where $\mathcal{AC}([0,1];\mathbb{R})$ is the space of absolutely continuous functions in $[0,1]$ and we recall that $q=\frac{1}{\sqrt{\pi}}\int_{\mathcal{Y}}\sigma^{2}(y) e^{-(y-m)^{2}}dy$.

A simple Lagrange multiplier argument shows that the variational problem in the display above can be solved explicitly, yielding
\[
S(x_{1})=\frac{(x_{1}-x_{0})^{2}}{2q}
\]
Therefore, for $x_{1}=x_{0}+\eta$ we get the logarithmic asymptotics
\[
\mathbb{P}_{x_{0},y_{0}}\left\{X_{t}\geq x_{0}+\eta\right\}\approx e^{-\frac{S(x_{0}+\eta)}{t}},\textrm{ as }t\downarrow 0.
\]
Consider now $\nu>0$ and formally set $\eta=\nu \sqrt{t}$. Then, we have
\[
S(x_{0}+\nu \sqrt{t})=\frac{\nu^{2}}{2q} t
\]
and notice that $S''(x_{0})=1/q$. This implies
\[
\mathbb{P}_{x_{0},y_{0}}\left\{\frac{X_{t}-x_{0}}{\sqrt{t}}\geq \nu\right\}\approx e^{-\frac{1}{2}S''(x_{0})\nu^{2}}=e^{-\frac{\nu^{2}}{2q}},\textrm{ as }t\downarrow 0.
\]
This is exactly the Gaussian limit law established in Corollary \ref{C:CLT2} with $\bar{c}=0$.

\section{Acknowledgements}
The author was partially supported, during revisions of this article, by the National Science Foundation
(DMS 1312124).


\begin{thebibliography}{0}





\bibitem {BLP}A. Bensoussan, J.L. Lions, G. Papanicolaou, \textit{Asymptotic
Analysis for Periodic Structures}, Vol 5, Studies in Mathematics and its
Applications, North-Holland Publishing Co., Amsterdam, 1978.


\bibitem{BaierFreidlin}
D. Baier and M. I. Freidlin, Theorems on large deviations and stability for random
perturbations, Dokl. Akad. Nauk SSSR Vol. 235 (1977), pp. 253-256.
= Soviet Math. Dokl. Vol. 18 (1977), pp. 905-909.

\bibitem{Billingsley1968}
P. Billingsley, \textit{Convergence of Probability Measures}, 1968, Wiley, New York.


\bibitem {DupuisSpiliopoulos}P. Dupuis and K. Spiliopoulos, Large deviations
for multiscale problems via weak convergence methods, \emph{Stochastic
Processes and their Applications}, Vol. 122, (2012), pp. 1947-1987.



\bibitem{Freidlin1978}
M.I. Freidlin, The Averaging Principle and Theorems on Large Deviations, \emph{Russian Mathematical Surveys} Vol. 33, No. 5, (1978), pp. 117-176.

\bibitem {FS}M.I. Freidlin, R. Sowers, A comparison of homogenization and large
deviations, with applications to wavefront propagation , \textit{Stochastic
Process and Their Applications}, Vol. 82, Issue 1, (1999), pp. 23--52.


\bibitem{FWBook}
 M.I. Freidlin, A.D. Wentzell, Random Perturbations of Dynamical Systems, 2nd Edition, 1998, Springer.

\bibitem{Guillin}
A. Guillin, Averaging principle of SDE with small diffusion: moderate deviations, {\em Annals of Probability}, Vol. 31, No. 1, (2003), pp. 413�443.

\bibitem{Hasminskii}
R.Z. Khasminskii, {\em Stochastic Stability of Differential Equations}, 2nd Ed., Springer, 2011.

\bibitem{Hollander2000}
F. Den Hollander, {\em Large deviations}, American Mathematical Society, Providence, RI, 2000.

\bibitem{KlebanerLipster}
Klebaner, F. C. and Liptser, R. Moderate Deviations for Randomly Perturbed Dynamical Systems, Stochastic Processes and their Applications, Vol. 80, (1999), pp. 157-176.

\bibitem{Liptser}
Liptser, R. Sh. and Shiryayev, A. N., Theory of martingales, Mathematics and its Applications (Soviet Series), 49, 1989.

\bibitem{LiptserPaper}
Liptser, Robert and Stoyanov, Jordan, \emph{Stochastic version of the averaging principle for diffusion type processes}, Stochastics Stochastics Rep., Stochastics and Stochastics Reports, Vol. 32, No. 3-4, (1990), pp. 145--163.


\bibitem{PardouxVeretennikov1}E. Pardoux, A.Yu. Veretennikov, On Poisson
equation and diffusion approximation I, \textit{Annals of Probability}, Vol.
29, No. 3, (2001), pp. 1061-1085.


\bibitem{PardouxVeretennikov2}E. Pardoux, A.Yu. Veretennikov, On Poisson
equation and diffusion approximation 2, \textit{Annals of Probability}, Vol.
31, No. 3, (2003), pp. 1166-1192.



\bibitem{Spiliopoulos2012}
K. Spiliopoulos,
\newblock Large deviations and importance sampling for systems of slow-fast motion,
  \newblock \emph{Applied Mathematics and Optimization,}, Vol. 67, (2013), pp. 123-161.


\bibitem {Veretennikov}A. Yu. Veretennikov, On large deviations in the
averaging principle for {SDEs} with a ``full dependence'', correction,
{\tt arXiv:math/0502098v1 [math.PR]} (2005). Initial
article in \textit{Annals of Probability}, Vol. 27, No. 1, (1999), pp. 284-296.

\bibitem {VeretennikovSPA2000}A. Yu. Veretennikov, On large deviations for
SDEs with small diffusion and averaging, \textit{Stochastic Processes and
their Applications}, Vol. 89, Issue 1, (2000), pp. 69-79.

\end{thebibliography}
\end{document}